\documentclass{amsart}

\usepackage{amsfonts,amssymb}
\usepackage{amsmath}

\newtheorem{thm}{Theorem}
\newtheorem{lem}[thm]{Lemma}
\newtheorem{cor}[thm]{Corollary}

\begin{document}

\title[Schr\"oder's method for the matrix $p$th root ]{A study of  Schr\"oder's method for the matrix $p$th root using power series expansions}

\author{Chun-Hua Guo}
\address{Department of Mathematics and Statistics, 
University of Regina, Regina,
SK S4S 0A2, Canada} 
\email{chun-hua.guo@uregina.ca}

\author{Di Lu}
\address{Department of Mathematics and Statistics, 
University of Regina, Regina,
SK S4S 0A2, Canada} 
\email{ludix203@uregina.ca}

\thanks{This work was supported in part by a grant from
the Natural Sciences and Engineering Research Council of Canada.}



\begin{abstract}

When $A$ is a matrix with all eigenvalues  in the disk $|z-1|<1$, the principal 
$p$th root of $A$ can be computed by Schr\"oder's method, among many other methods. 
In this paper we present a further study of Schr\"oder's method for the matrix $p$th root, through an examination of  power series expansions of some sequences of scalar functions.  
Specifically, we obtain a new and informative error estimate for the matrix sequence generated by the Schr\"oder's method,  a monotonic convergence result when $A$ is a nonsingular $M$-matrix, and  a structure preserving result when $A$ is a nonsingular $M$-matrix or a real nonsingular $H$-matrix with positive diagonal entries. 

\end{abstract}

\maketitle

\section{Introduction}

For a given integer $p\ge 2$  and a matrix  $A\in {\mathbb C}^{n\times n}$ whose eigenvalues are in the open disk $\{z\in \mathbb C: \ |z-1|<1\}$, 
the principal $p$th root of $A$ exists and is denoted by $A^{1/p}$ \cite{high08}. Various methods can be used to  compute $A^{1/p}$; see 
\cite{bhm05,gria10,guo10,guhi06,high08,hili11,hili13,iann06,iann07,iama13,lin10,lihu17,smit03,ziet14}. 

In this paper we are concerned with the Schr\"oder family of iterations, also called Schr\"oder's method for short, which is a special case of  
the dual Pad{\'e} family of iterations proposed in \cite{ziet14}. 
 
In the scalar case of computing $a^{1/p}$, the dual Pad{\'e} family of iterations has  the form 
\begin{equation}\label{pade}
x_{k+1}=x_{k}\frac{Q_{\ell m}(1-ax_{k}^{-p})}{P_{\ell m}(1-ax_{k}^{-p})}, \quad x_0=1, 
\end{equation}
where $P_{\ell m}(t)/Q_{\ell m}(t)$ is the $[\ell/m]$ Pad\'{e} approximant to the function $(1-t)^{-1/p}$, or equivalently 
$Q_{\ell m}(t)/P_{\ell m}(t)$ is the $[m/\ell]$ Pad\'{e} approximant to the function $(1-t)^{1/p}$.

When $\ell=m=1$, we get  Halley's method. When $\ell=0$, we get the Schr\"oder family of iterations. 
Within the Schr\"oder family, we get Newton's method when $m=1$, and get Chebyshev's method when $m=2$.

For $a=1-z$ with $|z|<1$,
Each $x_{k}$ from the dual Pad\'e iteration \eqref{pade}  has a power series expansion 
\begin{equation}\label{pades}
x_k=\sum_{i=0}^{\infty}c_{k,i}z^i, \quad |z|< 1. 
\end{equation}
It is conjectured in \cite{ziet14} that $c_{k,i}<0$ for $i\ge 1$ (as long as the series in \eqref{pades} is not reduced to a finite series). 
The conjecture is an extension of an earlier 
conjecture in \cite{guo10} for Newton's method and Halley's method and a similar conjecture in \cite{lin11} for Chebyshev's method. 

For $p=2$, the conjecture for Newton's method is shown to be true in \cite{guo10}, by using a result proved in 
\cite{mein04}, and a more direct proof is presented in \cite{koub12} for both Newton's method and Halley's method. 
For any integer $p\ge 2$, the conjecture  has been proved very recently \cite{lugu18} for both Newton's method and Halley's method. 
The conjecture for Chebyshev's method has remained open even for $p=2$. 

In this paper we will prove the conjecture for the whole Schr\"oder family (which include Chebyshev's method) for all $p\ge 2$. 
However, the conjecture is not true for the whole dual Pad\'e family. 
Indeed for $\ell=1$ and $m=0$, we have $P_{10}(t)=1+\frac{1}{p}t$ and $Q_{10}(t)=1$. 
It follows from \eqref{pade} that 
$$
x_1=\frac{1}{1+\frac{1}{p}(1-(1-z))}=\frac{1}{1+\frac{1}{p}z}=\sum_{i=0}^{\infty}(-1)^i\frac{1}{p^i}z^i. 
$$

From the Schr\"oder's method for computing $a^{1/p}$, we can get the corresponding Schr\"oder's method for computing $A^{1/p}$. 
In particular, we have Chebyshev's method for computing $A^{1/p}$. Chebyshev's method is called Euler's method in \cite{lihu17} and its efficiency 
(when properly implemented) has been shown in that paper. This has provided us additional motivation to further study Schr\"oder's method for the matrix $p$th root.

\section{Preliminaries} 

Schr\"oder's method for the matrix $p$th root will be studied through an examination of  power series expansions of some sequences of scalar functions.  

We start with the rising factorial notation 
$$
(x)_i=x(x+1)\cdots (x+i-1), 
$$
where $x$ is a real number and $i\ge 0$ is an integer. We have $(x)_0=1$ by convention. 

We have the binomial expansion 
\begin{equation}\label{be}
(1-t)^{1/p}=\sum_{i=0}^{\infty}b_{i}t^i, \quad |t|< 1, 
\end{equation}
where 
\begin{equation}\label{b0bi}
b_0=1, \quad b_i=
\frac{\left (-\frac{1}{p}\right )_i}{i!}<0, 
\quad i\ge 1. 
\end{equation}
By a limit argument, we can show that the equality in \eqref{be} also holds for $t=1$. So we have $\sum_{i=0}^{\infty}b_i=0$. 

Let $T_m(t)$ be the polynomial of degree $m$ 
$$T_{m}(t)=\sum_{i=0}^{m} b_i t^i, 
$$
which is the sum of the first $m+1$ terms in the power series \eqref{be}. 
It is readily seen that we have the Taylor expansion
\begin{equation}\label{sa}
(T_m(t))^{-p}=\sum_{i=0}^{\infty}a_{i}t^i, \quad |t|< 1, 
\end{equation}
where $a_0=1$ and $a_i>0$ for all $i\ge 0$. 

In the scalar case of computing $a^{1/p}$, we let the residual be
$R(x_k)= 1-ax_{k}^{-p}$,   
The Schr\"oder's iteration applied to the function $x^p-a$ gives (see \cite{calo11a}) the iteration 
\begin{equation}\label{schr}
x_{k+1}=x_{k}T_m (R(x_k)). 
\end{equation}
So indeed it is a special case of \eqref{pade} with $\ell=0$. 
When $m=2$, iteration \eqref{schr} is the same as Chebyshev's method \cite{cama90} applied to the function 
$x^p-a$, as noted in \cite{lin11}. 

For iteration \eqref{schr}, we have 
\begin{equation}\label{xxS}
x_k-x_{k+1}=x_k(1-T_m(R(x_k))=x_kR(x_k) \sum_{i=1}^m (-b_i)(R(x_k))^{i-1}, 
\end{equation}
and 
$$
R(x_{k+1})=f(R(x_k))
$$
with 
\begin{equation}\label{ffS}
f(t)=1-(T_m(t))^{-p}(1-t). 
\end{equation}

We have by \eqref{sa} that   $f(t)=\sum_{i=0}^{\infty}c_i t^i$ with $c_0=0$ and  $c_i=a_{i-1}-a_i$ for $i\ge 1$.
It is shown in \cite{calo11b}  that 
\begin{equation}\label{sumct}
f(t)=\sum_{i=0}^{\infty}c_i t^i=\sum_{i=m+1}^{\infty}c_i t^i
\end{equation}
with $c_i>0$ for $i\ge m+1$. 
This means that we actually have $a_0=a_1=\cdots = a_m=1$ in \eqref{sa}.

For $a=1-z$ with $|z|<1$, Schr\"oder's method for finding $(1-z)^{1/p}$ is 
\begin{equation}\label{schrz}
x_{k+1}=x_{k}T_m (1-(1-z)x_k^{-p})),  \quad x_0=1. 
\end{equation}
To emphasize the dependence of $x_k$ on $z$, we will write $x_k(z)$ 
for $x_k$. 
Each $x_k(z)$ has a power series expansion 
\begin{equation}\label{nke}
x_k(z)=\sum_{i=0}^{\infty}c_{k,i}z^i, \quad |z|< 1. 
\end{equation}

The following connection between Schr\"oder's method and the binomial expansion is included in the more general 
Theorem 6.1 of \cite{ziet14}; it can also be proved in the same way as \cite[Theorem 10]{guo10} is proved for Newton's method and Halley's method. 

\begin{thm}\label{thmNHb}
For Schr\"oder's method, $c_{k,i}=b_i$ for $k\ge 0$ and $0\le i\le (m+1)^{k}-1$. 
\end{thm}

From the theorem, we know that $c_{k,0}=1$ for all $k\ge 0$ and that  $c_{k,i}<0$ for $k\ge 1$ and $1\le i\le (m+1)^{k}-1$. 
To obtain some new results for Schr\"oder's method for the matrix $p$th root, we need to prove that $c_{k,i}\le 0$ for $k\ge 1$ and all $i\ge (m+1)^{k}$. 
When $k=1$, we have $x_1(z)=T_m(1-(1-z))=T_m(z)$, so $c_{1,i}=0$ for $i\ge m+1$.

\section{Sign pattern of coefficients $c_{k,i}$  in power series expansions} 

To determine the sign pattern of $c_{k,i}$ for Schr\"oder's method,  
we will show that $c_{k,i}$ decreases when $k$ increases (and $i\ge 1$ is fixed), as in \cite{lugu18} for Newton's method and Halley's method. In this process, we will need a useful recursion for the coefficients $a_i$ in \eqref{sa}, and some good luck as well!

For the scalar case of computing $(1-z)^{1/p}$, where $|z|<1$, the residual is $R(x_k)=1-(1-z)x_k^{-p}$. 
We have the following result. 

\begin{lem}\label{ppp0}
 The coefficients in the power series expansion (in the variable $z$) of $R(x_k)$ are all nonnegative. 
\end{lem}

\begin{proof}
We have $R(x_0)=z$. The result is proved by induction since 
$$
R(x_{k+1})=f(R(x_k))=\sum_{i=m+1}^{\infty}c_i (R(x_k))^i
$$
with $c_i>0$ for $i\ge m+1$. 
\end{proof}

We will show that the coefficients in the power series expansion of $x_kR(x_k)$ are also all nonnegative. 

Note that 
\begin{equation}\label{xkRk}
x_{k+1}R(x_{k+1})=x_kg(R(x_k)), 
\end{equation}
where 
$$
g(t)=T_m(t)  f(t). 
$$
It follows from \eqref{sumct} that $g(t)$ 
 has a Taylor expansion 
$g(t)=\sum_{i=m+1}^{\infty}d_it^i$, where 
$$
d_i=b_0c_i+b_1c_{i-1}+\cdots + b_m c_{i-m}
$$
for each $i\ge m+1$. 
We are going to prove that $d_i>0$ for all $i\ge m+1$. 
Since $b_0=1$ and $c_k=a_{k-1}-a_k$ for $k\ge 1$, we need to show for $i\ge m+1$ that 

\begin{equation}\label{arec}
-a_i+(b_0-b_1)a_{i-1}+(b_1-b_2)a_{i-2}+\cdots + (b_{m-1}-b_m)a_{i-m} + b_m a_{i-m-1}>0. 
\end{equation}

The following recursion about the coefficients $a_i$ in \eqref{sa} will play an important role. 
It is equation (2.11)  in \cite{calo11b}, with some notation changes.  

\begin{lem}\label{ca}
For each $k\ge m-1$, 
$$
a_{k+1}=\frac{1}{k+1}\sum_{s=0}^{m-1}(k-s+p(s+1))(-b_{s+1})a_{k-s}. 
$$
\end{lem}

By Lemma \ref{ca} with $k=i-1$, \eqref{arec} becomes
\begin{equation}\label{ineq8}
\sum_{s=0}^{m-1}\alpha_s a_{i-1-s}  + b_m a_{i-m-1}>0, 
\end{equation}
where for $s=0, 1, \ldots, m-1$
\begin{eqnarray*}
\alpha_s&=&b_s-b_{s+1}+\frac{1}{i}(i-1-s+p(s+1))b_{s+1}\\
&=&b_s+\frac{1}{i}(p-1)(s+1)b_{s+1}\\
&=&b_s+\frac{1}{i}(p-1)(s+1)b_{s}\frac{-\frac{1}{p}+s}{s+1}\\
&=&b_s\left (1+\frac{1}{i}(p-1)\left (-\frac{1}{p}+s\right )\right ).
\end{eqnarray*}
Thus \eqref{ineq8} is equivalent to 
\begin{equation}\label{ineq9}
\frac{p(i-1)+1}{pi}a_{i-1}+\sum_{s=1}^{m-1}b_s \left (1+\frac{1}{i}(p-1)\left (-\frac{1}{p}+s\right )\right )   a_{i-1-s}  + b_m a_{i-m-1}>0. 
\end{equation}
By Lemma \ref{ca} with $k=i-2$, we have 
\begin{eqnarray*}
a_{i-1}&=&\frac{1}{i-1}\sum_{s=0}^{m-1}(i-2-s+p(s+1))(-b_{s+1})a_{i-2-s}\\
&=&\frac{1}{i-1}\sum_{s=1}^{m}(i-1-s+ps)(-b_s)a_{i-1-s}.
\end{eqnarray*}
Thus \eqref{ineq9} becomes 
\begin{equation}\label{ineq10}
\sum_{s=1}^{m-1} \beta_s b_s a_{i-1-s} +\beta_m b_m a_{i-m-1}>0, 
\end{equation}
where 
$$
\beta_m=\frac{p(i-1)+1}{pi}\frac{-1}{i-1}(i-1-m+pm) + 1, 
$$ 
and for $s=1, \ldots, m-1$ 
$$
\beta_s=\frac{p(i-1)+1}{pi}\frac{-1}{i-1}(i-1-s+ps) + 1+\frac{1}{i}(p-1)\left (-\frac{1}{p}+s\right ). 
$$ 
Since $a_i>0$ for $i\ge 0$ and $b_i<0$ for $i\ge 1$, a sufficient condition for \eqref{ineq10} to hold is 
that, for $s=1, \ldots, m$,  $\beta_s<0$, or equivalently 
$-pi(i-1)\beta_s>0$. Luckily, the sufficient condition does hold. 
Indeed, 
$$
-pi(i-1)\beta_m=(p(i-1)+1)(i-1-m+pm)-pi(i-1)>p(i-1)i-pi(i-1)=0, 
$$
and for $s=1, \ldots, m-1$
\begin{eqnarray*}
-pi(i-1)\beta_s&=&  (p(i-1)+1)(i-1-s+ps)-pi(i-1)-(i-1)(p-1)(-1+ps)\\
&=&p(i-1)i+p(i-1)(-1+ps)-p(i-1)s+i-1-s+ps\\
& & -pi(i-1)-(i-1)(p-1)(-1+ps)\\
& = & (-1+ps)(p(i-1)-(i-1)(p-1))-p(i-1)s+i-1-s+ps\\
& = & (-1+ps)(i-1)-p(i-1)s+i-1-s+ps\\
&=&(p-1)s>0.
\end{eqnarray*}
Therefore, \eqref{arec} holds for all $i\ge m+1$. We have thus proved the following result. 

\begin{lem}\label{Hdp}
The function $g(t)$ has a Taylor expansion $g(t)=\sum_{i=m+1}^{\infty}d_it^i$ with $d_i>0$ for $i\ge m+1$. 
\end{lem}

We are now ready to prove the following result. 

\begin{lem}\label{xRxH}
 For each $k\ge 0$, the coefficients in the power series expansion of $x_kR(x_k)$ are all nonnegative. 
\end{lem}

\begin{proof}
We have $x_0R(x_0)=z$ and 
\begin{equation*}
x_{k+1}R(x_{k+1})=x_k g(R(x_k))=x_k R(x_k)\sum_{i=m+1}^{\infty}d_i (R(x_k))^{i-1}. 
\end{equation*}
The result is then proved by induction, using Lemmas \ref{ppp0} and \ref{Hdp}. 
\end{proof}

For Schr\"oder's method, we have $x_k(z)=\sum_{i=0}^{\infty}c_{k,i}z^i$ with $c_{k,0}=1$ for all $k\ge 0$. 
We already know that $c_{0,i}=0$ for all $i\ge 1$,  $c_{1,i}<0$ for $1\le i\le m$ and $c_{1,i}=0$ for $i\ge  m+1$. 
The next result determines the sign pattern of $c_{k,i}$ for $k\ge 2$ and $i\ge 1$. 

\begin{thm}\label{Hd}
For each $i\ge 1$, $c_{k,i}$ decreases as $k$ increases and becomes equal to $b_i$ for all $k$ sufficiently large. In particular,  $c_{k,i}\le 0$ for $k\ge 0$ and $i\ge 1$. 
Moreover, $c_{k,i}<0$ for $k\ge 2$ and $i\ge 1$. 
\end{thm}

\begin{proof}
That  $c_{k,i}$ decreases as $k$ increases  follows directly from \eqref{xxS}, Lemma \ref{ppp0}, Lemma \ref{xRxH}, and the fact that $b_i<0$ for $i\ge 1$. 
We know from Theorem \ref{thmNHb} that $c_{k,i}=b_i$ when $i\le (m+1)^{k}-1$, i.e., when $k\ge \ln(i+1)/\ln(m+1)$. 
Since $c_{0,i}=0$ for all $i\ge 1$, the monotonicity of $c_{k,i}$ implies that  $c_{k,i}\le 0$ for $k\ge 0$ and $i\ge 1$. 
To show $c_{k,i}<0$ for $k\ge 2$ and $i\ge 1$, we only need to show that $c_{2,i}<0$ for $i\ge 1$. 
By \eqref{xxS} and \eqref{xkRk}
\begin{eqnarray*}
 x_1-x_{2}&=&x_1R(x_1) \sum_{i=1}^m (-b_i)(R(x_1))^{i-1}\\
&=&x_0 g(R(x_0))\sum_{i=1}^m (-b_i)(R(x_1))^{i-1}\\
&=&g(z)\sum_{i=1}^m (-b_i)(R(x_1))^{i-1}. 
\end{eqnarray*}
It follows from Lemmas \ref{ppp0} and \ref{Hdp} that 
$$
x_1-x_2=\sum_{i=m+1}^{\infty}e_i z^i
$$
with $e_i\ge d_i (-b_1)>0$ for $i\ge m+1$. 
Therefore,  for $i\ge m+1$ we have $c_{2,i}<c_{1,i}= 0$ and  for $1\le i\le m$ we have $c_{2,i}=c_{1,i}< 0$. 
We have thus proved that  $c_{2,i}<0$ for $i\ge 1$. 
\end{proof}

\section{Schr\"oder's method for the matrix $p$th root}

In the matrix case,  Schr\"oder's  method for finding $A^{1/p}$ is given by
\begin{equation}\label{SI}
X_{k+1}=X_k T_m(R(X_k)), \quad X_0=I,  
\end{equation}
where 
$$
R(X_k)=I-AX_k^{-p}. 
$$
Note that we have $X_kA=AX_k$ for Schr\"oder's  method  whenever $X_k$ is defined. 
Using this commutativity and its consequences, we immediately get the following result  from Theorem \ref{thmNHb}.  

\begin{thm}\label{thm2} 
Suppose that all eigenvalues of $A$ are in $\{z: |z-1|< 1\}$ and 
write $A=I-B$ 
$($so $\rho(B)<1$$)$. Let $(I-B)^{1/p}=\sum_{i=0}^{\infty}b_iB^i$ be 
the binomial expansion (where the coefficients $b_i$ are given by \eqref{b0bi}). 
Then the sequence $\{X_k\}$ generated by Schr\"oder's method 
has the power series expansion $X_k=\sum_{i=0}^{\infty}c_{k,i}B^i$, 
with  $c_{k,i}=b_i$ for $i=0, 1, \ldots, (m+1)^k-1$. 
\end{thm}

Recall that $\sum_{i=0}^{\infty}b_i=0$ for the coefficients $b_i$. Let 
$s_k=\sum_{i=0}^{k-1}b_i$.
Then 
$s_k=\sum_{i=k}^{\infty}(-b_i)$,  $s_k\le 1$ for all $k\ge 1$,  and 
$\lim_{k\to \infty}s_k=0$. 

By using Theorems \ref{thm2} and \ref{Hd}, we can get the following nice error estimate. 

\begin{thm}\label{thm3} 
Suppose that all eigenvalues of $A$ are in $\{z: |z-1|< 1\}$ and 
write $A=I-B$. Then, for any matrix norm such that $\|B\|<1$, 
 the sequence $\{X_k\}$ generated by Schr\"oder's method satisfies 
$$
\|X_k-A^{1/p}\|\le \|B\|^{(m+1)^k}. 
$$
\end{thm}

\begin{proof}
We have by Theorem \ref{thm2} that
$$
X_k-A^{1/p}=\sum_{i=(m+1)^k}^{\infty}(c_{k,i}-b_i)B^i. 
$$
By Theorem \ref{Hd}, we have $0\le  c_{k,i}-b_i\le -b_i$ ($k\ge 0$, $i\ge 1$). 
It follows that 
\begin{eqnarray*}
\|X_k-A^{1/p}\|&\le& \sum_{i=(m+1)^k}^{\infty}(-b_i)\|B\|^i\le \sum_{i=(m+1)^k}^{\infty}(-b_i)\|B\|^{(m+1)^k}\\
&=&s_{(m+1)^k}\|B\|^{(m+1)^k}\le \|B\|^{(m+1)^k}. 
\end{eqnarray*}
\end{proof}

Note that we have actually given a sharper upper bound in the proof.

We now consider the computation of $A^{1/p}$, where $A$ is a nonsingular $M$-matrix or 
a real nonsingular $H$-matrix with positive diagonal entries.
It is known \cite{ando80,fisc83,john82} that $A^{1/p}$ is a nonsingular $M$-matrix for every nonsingular $M$-matrix $A$. 
It has been proved in \cite{guo10} that
when $A$ is a real nonsingular $H$-matrix with positive diagonal 
entries, so is  $A^{1/p}$.

As in \cite{guo10}, we let ${\mathcal M}_1$ be the set of all nonsingular $M$-matrices whose diagonal entries are in $(0, 1]$, 
and ${\mathcal H}_1$ be the set of all
real nonsingular $H$-matrices whose diagonal entries are in $(0, 1]$. 

We assume 
$A=I-B$ is in ${\mathcal M}_1$ (so $B\ge 0$) or ${\mathcal H}_1$. We can see from the binomial expansion that 
$A^{1/p}\in {\mathcal M}_1$ when $A\in {\mathcal M}_1$ and that $A^{1/p}\in {\mathcal H}_1$ when $A\in {\mathcal H}_1$. 
To find $(I-B)^{1/p}$ we generate a sequence $\{X_k\}$ by Schr\"oder's method, with $X_0=I$.

When $A\in {\mathcal M}_1$, we have the following monotonic convergence result. 

\begin{thm}\label{mono}
Suppose $A\in {\mathcal M}_1$. 
Then the sequence $\{X_k\}$ generated by Schr\"oder's 
method is monotonically decreasing and converges to $A^{1/p}$. 
\end{thm}

\begin{proof}
When $A\in {\mathcal M}_1$,  we can  write $A=I-B$ with $B\ge 0$ and $\rho(B)<1$. 
So all eigenvalues of $A$ are  in the open disk $\{z: \ |z-1|<1\}$. 
The convergence of $X_k$ to $A^{1/p}$ is known from Theorem \ref{thm3} for example. By Theorem \ref{thm2}, 
the sequence $\{X_k\}$ generated by Schr\"oder's 
method has the power series expansion $X_k=\sum_{i=0}^{\infty}c_{k,i}B^i=I+ \sum_{i=1}^{\infty}c_{k,i}B^i$. 
 It follows from Theorem \ref{Hd} that 
$X_k\ge X_{k+1}$ for all $k\ge 0$. 
\end{proof}

The following structure-preserving property of Schr\"oder's method follows readily from the above theorem. 

\begin{cor}
Let $A$ be in ${\mathcal M}_1$ and $\{X_k\}$ be generated by 
Schr\"oder's method. 
Then for all $k\ge 0$, $X_k$ are in ${\mathcal M}_1$. 
\end{cor}

\begin{proof}
We know that $A^{1/p}$ is in ${\mathcal M}_1$. For each $k\ge 0$, $X_k $ is a $Z$-matrix and $X_k\ge A^{1/p}$ by Theorem \ref{mono}. 
So $X_k$ is an $M$-matrix. The diagonal entries of $X_k$ are in $(0, 1]$ since $X_k\le X_0=I$. 
\end{proof}

We also have the following structure-preserving property. 

\begin{thm}
Let $A$ be in ${\mathcal H}_1$  and $\{X_k\}$ be generated by 
Schr\"oder's method. 
Then for all $k\ge 0$, $X_k$ are in  ${\mathcal H}_1$. 
\end{thm}

\begin{proof}
The result can be  shown as  in \cite{guo10}, using Theorem \ref{Hd}. 
\end{proof}

\section*{Acknowledgments} 

The research of Chun-Hua Guo was supported in part by a grant from the Natural Sciences and Engineering Research Council of Canada.

\end{document}